\documentclass[11pt,leqno]{amsart}
\usepackage{amsmath,amssymb,amsthm}
\usepackage{hyperref}

\usepackage{bbm}

\DeclareMathOperator{\co}{co}

\renewcommand{\geq}{\geqslant}
\renewcommand{\leq}{\leqslant}

\newcommand{\NA}{\operatorname{NA}}

\newcommand{\supp}{\operatorname{supp}}

\newtheorem{theorem}{Theorem}[section]
\newtheorem{lemma}[theorem]{Lemma}

\newtheorem{corollary}[theorem]{Corollary}
\theoremstyle{definition}

\theoremstyle{remark}
\newtheorem{remark}[theorem]{Remark}
\numberwithin{equation}{section}

\def\fnote#1{\footnote}

\def\ignora#1{}
\def\n3#1{\left\vert  \! \left\vert \! \left\vert \, #1 \, \right\vert \!
  \right\vert \! \right\vert }

\newcommand{\pten}{\ensuremath{\widehat{\otimes}_\pi}}

\begin{document}

\title{Several remarks on norm attaining in tensor product spaces}

\author{Abraham Rueda Zoca}
\address{Universidad de Murcia, Departamento de Matem\'{a}ticas, Campus de Espinardo 30100 Murcia, Spain
	\newline
	\href{https://orcid.org/0000-0003-0718-1353}{ORCID: \texttt{0000-0003-0718-1353} }}
\email{\texttt{abraham.rueda@um.es}}
\urladdr{\url{https://arzenglish.wordpress.com}}

\thanks{The research of Abraham Rueda Zoca was supported by MCIN/AEI/10.13039/501100011033: Grants MTM2017-86182-P, PGC2018-093794-B-I00, PID2021-122126NB-C31 and PID2021-122126NB-C32;  by Fundaci\'on S\'eneca: ACyT Regi\'on de Murcia grant 20797/PI/18, and by Junta de Andaluc\'ia: Grants A-FQM-484-UGR18 and by FQM-0185.}

\subjclass[2020]{46B20, 46B28}

\keywords {Projective tensor product; strong subdifferentiability; $w^*$-Kadec-Klee property;  norm-attaining tensor }

\maketitle

\begin{abstract}
The aim of this note is to obtain results about when the norm of a projective tensor product is strongly subdifferentiable. We prove that if $X\pten Y$ is strongly subdifferentiable and either $X$ or $Y$ has the metric approximation property then every bounded operator from $X$ to $Y^*$ is compact. We also prove that $(\ell_p(I)\pten \ell_q(J))^*$ has the $w^*$-Kadec-Klee property for every non-empty sets $I,J$ and every $2<p,q<\infty$, obtaining in particular that the norm of the space $\ell_p(I)\pten \ell_q(J)$ is strongly subdifferentiable. This extends several results of Dantas, Kim, Lee and Mazzitelli. We also find examples of spaces $X$ and $Y$ for which the set of norm-attaining tensors in $X\pten Y$ is dense but whose complement is dense too. 
\end{abstract}

\section{Introduction}

The study of norm attaining functionals has been a long standing topic in functional analysis because it has shown to have strong connections with the structure of the underlying space. Probably the best example of this is the classical result of James which says that a Banach space $X$ is reflexive if, and only if, every linear continuous functional attains its norm \cite[Corollary 3.56]{checos}. Another example is the celebrated result due to Bishop and Phelps, which says that the set of norm attaining functionals is always dense \cite[Theorem 3.54]{checos}. The relevance of norm attaining elements opened the door, from the seminal paper of J. Lindenstrauss \cite{linds2}, to study the problem of when the set of norm attaining elements are dense for other kind of mappings such as bounded linear operators \cite{B,linds2,martin16}, bounded multilinear mappings \cite{gasp,Choi}, polynomials \cite{AGM,ADM} or Lipschitz mappings \cite{CGMR,godsurvey,kms}.

In the context of bilinear mappings, S. Dantas, S. K. Kim, H. J. Lee and M. Mazzitelli recently considered a new property related to norm-attainment: according to \cite[Definition 2.1]{dklm}, given two Banach spaces $X,Y$, we say that the pair $(X,Y)$ has the \textit{$L_{p,p}$ for bilinear mappings} if, given $\varepsilon>0$ and $(x,y)\in S_X\times S_Y$, there exists $\eta>0$ (which depends on $\varepsilon$ and on the pair $(x,y)$) satisfying that if a bilinear mapping $B:X\times Y\longrightarrow \mathbb R$  with $\Vert B\Vert=1$ satisfies $B(x,y)>1-\eta$ then there exists another bilinear mapping $G:X\times Y\longrightarrow \mathbb R$ with $\Vert G\Vert=1$, $G(x,y)=1$ and $\Vert B-G\Vert<\varepsilon$.

The interest on this property is double. On the one hand, this definition is a natural generalisation of the Bishop-Phelps-Bollob\'as property for bilinear mapping studied, for instance, in \cite{abgm}. On the other hand, it is a natural version for bilinear mapping of the classical characterisation of strongly subdifferentiable norms (see the formal definition in Section \ref{section:SSD}) given in \cite{fp}: a Banach space $X$ is strongly subdifferentiable  if, and only if, for every $\varepsilon>0$ and $x\in S_X$ there exists $\eta>0$ satisfying that if $y^*\in S_{X^*}$ verifies $y^*(x)>1-\eta$ then there exists $x^*\in S_{X^*}$ with $x^*(x)=1$ and $\Vert x^*-y^*\Vert<\varepsilon$.

Taking into account the isometric isomorphism $B(X\times Y)=(X\pten Y)^*$ coming from classical tensor product theory (see below), it is natural to think that the property $L_{p,p}$ for a pair $(X,Y)$ may be related with the strong subdifferentiability of the norm of $X\pten Y$. It turns out that if $X\pten Y$ is strongly subdifferentiable then the pair $(X,Y)$ has the $L_{p,p}$ for bilinear mappings (see e.g. \cite[Proposition 4.2]{djmr}) but the converse is not true \cite[Theorem]{dklm} as $\ell_2\pten \ell_2$ is a counterexample.

The aim of Section \ref{section:SSD} is to go further in showing that strong subdifferentiability is very restrictive in a projective tensor product. Indeed, we prove in Theorem \ref{theo:nececompact} that given two Banach spaces $X$ and $Y$ such that $X$ or $Y$ has the metric approximation property, if $X\pten Y$ is strongly subdifferentiable then every bounded operator from $X$ to $Y^*$ must be compact. This explain why $\ell_2\pten \ell_2$ is not strongly subdifferentiable and shows that this property must be seeked in a restrictive class of projective tensor product spaces. In the search of positive examples we look at the known result that $\ell_p\pten \ell_q$ is strongly subdifferentiable if $2<p,q<\infty$ \cite[Corollary 2.8]{dklm}. Indeed, this result relies on a nice one of S. J. Dilworth and D. Kutzarova \cite[Theorem 4]{dk}, which asserts that, for $2<p,q<\infty$, the space $(\ell_p\pten \ell_q)^*$ enjoys the \textit{$w^*$-Kadec-Klee property} (see the definition before Theorem \ref{theo:SSDnoseparable}). Our next aim in Theorem \ref{theo:SSDgenerallplq} is to extend the above mentioned result to arbitrary density characters by proving that for $2<p,q<\infty$ the space $(\ell_p(I)\pten \ell_q(J))^*$ has the $w^*$-Kadec-Klee property obtaining, as a consequence, that $\ell_p(I)\pten \ell_q(J)$ is strongly subdifferentiable. As a consequence of this result, we are able to prove that the pair $(\ell_p(I),\ell_q(J))$ has the $L_{p,p}$ for bilinear mappings, which improves \cite[Theorem 2.7 (a)]{dklm}.

In Section \ref{section:naprojtensor} we consider a quite recent concept of norm attainment related to nuclear operators (see \cite[Section 2.3]{djrr21}). According to \cite[Definition 2.1]{djrr21}, an element $z\in X\pten Y$ is said \textit{attain its projective norm} if there exists a sequence $(x_n)$ in $X$ and $(y_n)$ in $Y$ such that $\Vert u\Vert=\sum_{n=1}^\infty \Vert x_n\Vert \Vert  y_n\Vert$ and $u=\sum_{n=1}^\infty x_n\otimes y_n$. We denote $\NA_\pi(X\pten Y)$ the set of those $z$ which attains its nuclear norm.

In the sucessive papers \cite{dgjr22,djrr21} a lot of examples of Banach spaces $X$ and $Y$ are exhibited so that $\NA(X\pten Y)$ is dense in $X\pten Y$. It is also known that there are examples where $\NA_\pi(X\pten Y)\neq X\pten Y$, and even it is known that $\NA_\pi(X\pten Y)$ may fail to be dense (see \cite[Theorem 5.1]{djrr21}). A natural question in this line is whether $(X\pten Y)\setminus \NA_\pi(X\pten Y)$ may be dense. This can be compared with the study of when the non norm attaining linear functionals may be dense in a Banach space, a problem which has been considered in the literature (as a matter of example, let us point out that in \cite{acka} it is proved that every non-reflexive Banach space admits an equivalent renormig such that the set of non norm-attaining linear functionals is dense).

In Theorem \ref{theo:nonnormattainfincor} we prove that if $X$ is an infinite dimensional Banach space whose norm depends upond finitely many coordinates and $Y$ is an infinite dimensional Banach space then $\NA_\pi(X\pten Y)$ is contained in $X\otimes Y$. As a consequence, we get in Theorem \ref{theo:densiyintevacio} many examples of $X$ and $Y$ for which $\NA_\pi(X\pten Y)$ and its complement are dense.

\textbf{Terminology:} We will consider for simplicity real Banach spaces. We denote by $B_X$ and $S_X$ the closed unit ball and the unit sphere, respectively, of the Banach space $X$. We denote by $L(X, Y)$ the set of all bounded linear operators from $X$ into $Y$. If $Y = \mathbb R$, then $L(X, \mathbb R)$ is denoted by $X^*$, the topological dual space of $X$. We denote by $B(X \times Y)$ the Banach space of bounded bilinear mappings from $X \times Y$ into $\mathbb R$. It is well-known that the space $B(X \times Y)$ and $L(X, Y^*)$ are isometrically isomorphic as Banach spaces. We denote by $K(X, Y)$ the set of all compact operators and by $F(X, Y)$ the space of all operators of finite-rank from $X$ into $Y$.

The projective tensor product of $X$ and $Y$, denoted by $X \pten Y$, is the completion of the algebraic tensor product $X \otimes Y$ endowed with the norm
$$
\|z\|_{\pi} := \inf \left\{ \sum_{i=1}^n \|x_i\| \|y_i\|: z = \sum_{i=1}^n x_i \otimes y_i \right\},$$
where the infimum is taken over all such representations of $z$. The reason for taking completion is that $X\otimes Y$ endowed with the projective norm is complete if, and only if, either $X$ or $Y$ is finite dimensional (see \cite[P.43, Exercises 2.4 and 2.5]{ryan}).

It is well-known that $\|x \otimes y\|_{\pi} = \|x\| \|y\|$ for every $x \in X$, $y \in Y$, and the closed unit ball of $X \pten Y$ is the closed convex hull of the set $B_X \otimes B_Y = \{ x \otimes y: x \in B_X, y \in B_Y \}$. Throughout the paper, we will make use of both formulas indistinctly, without any explicit reference.

Observe the action of an operator $G: X \longrightarrow Y^*$ as a linear functional on $X \pten Y$ is given by
$$
G \left( \sum_{n=1}^{k} x_n \otimes y_n \right) = \sum_{n=1}^{k} G(x_n)(y_n),$$
for every $\sum_{n=1}^{k} x_n \otimes y_n \in X \otimes Y$. This action establishes a linear isometry from $L(X,Y^*)$ onto $(X\pten Y)^*$ (see e.g. \cite[Theorem 2.9]{ryan}). All along this paper we will use the isometric identification $(X\pten Y)^*=L(X,Y^*)=B(X\times Y)$ without any explicit mention.

From the equality $B_{X\pten Y}=\overline{\co}(B_X\otimes B_Y)$ and by Banach-Dieudonn\'e theorem \cite[Theorem 4.44]{checos} it is not difficult to prove that a bounded net $T_s\in L(X,Y^*)$ converges in the $w^*$ topology of $L(X,Y^*)=(X\pten Y)^*$ to some $T\in L(X,Y^*)$ if, and only if, $T_s(x)(y)=T_s(x\otimes y)\rightarrow T(x\otimes y)=T(x)(y)$ for every $x\in B_X$ and $y\in B_Y$. 

Observe that, given two Banach spaces $X$ and $Y$, the Banach space $X$ can be seen as an isometric subspace of $X\pten Y$. Indeed, given $y_0\in S_Y$, the bounded operator 
$$\begin{array}{ccc}
X & \longrightarrow & X\pten Y\\
x & \longmapsto & x\otimes y_0
\end{array}$$
is an isometry. 

Observe also that given two bounded operators $T:X\longrightarrow Z$ and $S:Y\longrightarrow W$, we can define an operator $T\otimes S:X\pten Y\longrightarrow Z\pten W$ by the action $(T\otimes S)(x\otimes y):=T(x)\otimes S(y)$ for $x\in X$ and $y\in Y$. It follows that $\Vert T\otimes S\Vert=\Vert T\Vert\Vert S\Vert$.

As an easy consequence, if $Z\subseteq X$ is a $1$-complemented subspace, then $Z\pten Y$ is a $1$-complemented subspace of $X\pten Y$ in the natural way (see \cite[Proposition 2.4]{ryan} for details).

Recall that a Banach space $X$ has the \emph{metric approximation property} (MAP)
if there exists a net $(S_\alpha)$ in $F(X,X)$ with $\Vert S_\alpha\Vert\leq 1$ for every $\alpha$ and such that
$S_\alpha (x) \to x$ for all $x \in X$.

\section{On the strong subdifferentiability}\label{section:SSD}

Recall that the norm of a Banach space $X$ is said to be \textit{strongly subdifferentiable (SSD)} if, for every $x\in S_X$, the one-sided limit
$$\lim\limits_{t\rightarrow 0^+}\frac{\Vert x+th\Vert-\Vert x\Vert}{t}$$
exists uniformly for $h\in S_X$. Observe that the norm of a Banach space $X$ is SSD if, and only if, for every $\varepsilon>0$ and $x\in S_X$ there exists $\eta>0$ satisfying that if $y^*\in S_{X^*}$ verifies $y^*(x)>1-\eta$ then there exists $x^*\in S_{X^*}$ with $x^*(x)=1$ and $\Vert x^*-y^*\Vert<\varepsilon$. See \cite{fp,gmz} and references therein for examples and background on the topic.

Let $X,Y$ be two Banach spaces. Observe that in order to $X\pten Y$ being SSD then $X$ and $Y$ must be SSD because the property of being SSD is inherited by subspaces (it is clear and explicitly mentioned in \cite[Section 2]{fp}) and because $X$ and $Y$ are isometrically isomorphic to a subspace of $X\pten Y$. However, this necessary condition is far from being enough. Indeed, in \cite[Corollary 2.8]{dklm} it is observed that $\ell_p\pten \ell_q$ can not be SSD if $\frac{1}{p}+\frac{1}{q}\geq 1$ as $\ell_p\pten \ell_q$ contains an isometric copies of $\ell_1$ and, in particular, $\ell_p\pten \ell_q$ is not Asplund and consequently it can not be SSD \cite[Theorem 2 (i)]{gmz}.

The following result widely generalises the above mentioned result and exhibits a structural necessary condition for a projective tensor product to be SSD.

\begin{theorem}\label{theo:nececompact}
Let $X$ and $Y$ be two Banach spaces. Assume that either $X$ or $Y$ has the MAP. If $X\pten Y$ is SSD then every operator $T:X\longrightarrow Y^*$ is compact.

In particular, if $X$ and $Y$ are reflexive, if $X\pten Y$ is SDD then $X\pten Y$ is reflexive.
\end{theorem}

\begin{proof}
Since $X\pten Y$ is SSD, then $L(X,Y^*)=(X\pten Y)^*$ has no proper $1$-norming closed subspace \cite[Lemma 3]{gmz}. Since $X$ or $Y$ has the MAP we get that $K(X,Y^*)$ is $1$-norming for $X\pten Y$ (see e.g. \cite[Proposition 2.3]{llr2}) and it is closed. By the above, we derive that $L(X,Y^*)=K(X,Y^*)$, as desired.

The particular case of $X$ and $Y$ being reflexive follows from a well known characterisation of reflexivity of projective tensor product (see e.g. \cite[Theorem 4.21]{ryan}).
\end{proof}

This result recovers the fact that $\ell_p\pten \ell_q$ fails to be SSD when $\ell_p\pten \ell_q$ is not reflexive. Indeed, in this case we have the formal identity $i:\ell_p\longrightarrow \ell_{q^*}$ where the above theorem applies. However, we have more examples.

\begin{corollary}\label{cor:tensorxx*}
Let $X$ be an infinite dimensional Banach space with the MAP. For every infinite dimensional subspace $Y$ of $X$ we have that $Y\pten X^*$ is not SSD. In particular, $X\pten X^*$ is not SSD.
\end{corollary}

\begin{proof} Taking the inclusion operator $i:Y\longrightarrow X$, which is not compact since $i$ is an isometry and $Y$ is infinite dimensional, Theorem \ref{theo:nececompact} applies.
\end{proof}

Notice that Theorem \ref{theo:nececompact} reveals that SSD on a projective tensor product $X\pten Y$ impose severe restrictions on the space $L(X,Y^*)$ under the MAP assumption, which explains the big absense of examples of SSD projective tensor product spaces. Let us notice, however, that in the case when $L(X,Y^*)=K(X,Y^*)$, still few examples are known to be SSD. The reason is that the characterisation of the SSD implies the necessity of dealing with perturbation of operators, which is difficult even for finite-rank operators. Because of that, in practice, the existing examples of SSD projective tensor product spaces have been obtained by indirect arguments.

To the best of the author knownledge, the only known results about SSD in projective tensor products are the following ones.

\begin{enumerate}
\item $\ell_1^N\pten X$ is SSD if, and only if, $X$ is SSD \cite[Theorem C]{djmr}. This result follows because in this case $\ell_1^N\pten X=\ell_1^N(X)$ isometrically and by the characterisation of SSD norms in $\ell_1$-sums of spaces given in \cite[Proposition 2.2]{fp}.
\item $\ell_p\pten \ell_q$ is SSD if $2<p,q<\infty$ \cite[Corollary 2.8 (a)]{dklm}. This result follows since $(\ell_p\pten \ell_q)^*$ has the $w^*$-Kadec-Klee property in this case \cite[Theorem 4]{dklm} and because, if $X$ is a reflexive Banach space such that $X^*$ has the $w^*$-Kadec-Klee property then the norm of $X$ is SSD (see the proof of \cite[Theorem 2.7]{dklm}).
\end{enumerate}

Our aim is to extend the above result to arbitrary $\ell_p(I)\pten \ell_q(J)$ for $2<p,q<\infty$. This will be done by proving that the dual has the $w^*$-sequential Kadec-Klee property, which will improve \cite[Theorem 4]{dk} to the non-separable case. In order to do so, let us introduce a bit of notation. Following \cite[Section 1]{dk}, we say that the dual of a Banach space $X$ has the \textit{$w^*$-Kadec-Klee property} if whenever $(x_n^*)$ is a sequence in $S_{X^*}$ satisfying that $x_n^*\rightarrow x^*\in S_{X^*}$ then $\Vert x_n^*-x^*\Vert\rightarrow 0$. See \cite{dk} for background and examples of spaces with the $w^*$-Kadec-Klee property.

Our interest in the $w^*$-Kadec-Klee property comes from (the proof of) \cite[Theorem 2.7]{dklm}, where it is proved that if $X$ is a Banach space such that $X^*$ has the $w^*$-Kadec-Klee property then the norm of $X$ is SSD.

\begin{theorem}\label{theo:SSDnoseparable}
Let $I, J$ be two infinite sets. Then $(\ell_p(I)\pten \ell_q(J))^*$ has the $w^*$-Kadec Klee property.
\end{theorem}

\begin{proof}
Take a sequence $T_n\in (\ell_p(I)\pten \ell_q(J))^*=L(\ell_p(I),\ell_{q^*}(J))$ such that $\Vert T_n\Vert=1$ for every $n$ and $T_n\rightarrow T\in S_{L(\ell_p(I),\ell_{q^*}(J))}$ in the $w^*$-topology of $(X\pten Y)^*$. Let us prove that $\Vert T_n-T\Vert\rightarrow 0$. 

To this end, we can assume with no loss of generality that $T_n$ is finite-rank for every $n\in\mathbb N$ because $L(\ell_p(I),\ell_{q^*}(J))=K(\ell_p(I),\ell_{q^*}(J))$ \cite[Theorem A2]{rosenthal} and then finite-rank operators are norm dense since $\ell_p(I)$ has the MAP and \cite[Proposition 4.12]{ryan} applies.

Hence we can write $T_n:=\sum_{k=1}^{p_n} x_{k,n}^*\otimes y_{k,n}^*$ for certain $p_n\in \mathbb N, x_{k,n}^*\in \ell_{p^*}(I)$ and $y_{k,n}^*\in \ell_{q^*}(J)$. Since every $x_{k,n}^*$ and $y_{k,n}^*$ have countable support we can find countable subsets $N\subseteq I$ and $M\subseteq J$ such that $\supp(x_{k,n}^*)\subseteq N$ and $\supp (y_{k,n}^*)\subseteq M$ holds for every $n\in\mathbb N$ and $k\in\{1,\ldots, p_n\}$. Set $i: \ell_p(N)\hookrightarrow \ell_p(I)$ and $j:\ell_q(M)\hookrightarrow \ell_q(J)$ the natural inclusion operators, and set $P:\ell_p(I)\longrightarrow \ell_p(N)$ and $Q:\ell_q(J)\longrightarrow \ell_q(M)$ the canonical (norm-one) projections. Given $n\in\mathbb N$ observe that, since $\supp(x_{k,n}^*)\subseteq N$ and $\supp(y_{k,n})\subseteq M$, we have that $T_n(x\otimes y)=T_n(i(P(x))\otimes j(Q(y)))$. Since $T_n\rightarrow T$ in the $w^*$ topology we conclude that $T(x\otimes y)=T(i(P(x))\otimes i(Q(y)))$ for every $x\in \ell_p(I)$ and $y\in \ell_q(J)$. Now define $G_n:=T_n\circ (i\otimes j)$ and $T:=T\circ (i\otimes j)$, which are elements of $(\ell_p(N)\pten \ell_q(M))^*$. 

We claim that $G_n\rightarrow G$ in the $w^*$ topology and that $\Vert G_n\Vert=\Vert G\Vert=1$ for every $n$. Let us prove first that $G$ is norm-one (the case of $G_n$ is similar). On the one hand, given $x\in S_{\ell_p(N)}, y\in S_{\ell_q(M)}$ we have
$$G(x\otimes y)=T(i(x)\otimes j(y))\leq \Vert T\Vert \Vert i(x)\otimes j(y)\Vert=1$$
since $\Vert T\Vert=1$ and $\Vert i\otimes j\Vert= \Vert i\Vert \Vert j\Vert=1$. For the reverse inequality take $\varepsilon>0$ and, since $\Vert T\Vert=1$, we can find $x\in S_{\ell_p(I)}$ and $y\in S_{\ell_q(J)}$ such that $T(x\otimes y)>1-\varepsilon$. Since $T(x\otimes y)=T(i(P(x))\otimes j(Q(y)))$ we have that
\[\begin{split}
1-\varepsilon<T(x\otimes y)=T(i(P(x))\otimes j(Q(y)))& =G(P(x)\otimes Q(y))\\
& \leq \Vert G\Vert \Vert P\otimes Q\Vert \Vert x\otimes y\Vert\\
& =\Vert G\Vert.
\end{split}\]
Since $\varepsilon>0$ was arbitrary we conclude that $\Vert G\Vert=1$. The same argument proves that $\Vert G_n\Vert=1$ holds for every $n\in\mathbb N$.

Now let us prove that $G_n\rightarrow G$ in the $w^*$-topology of $(\ell_p(N)\pten \ell_q(M))^*$. Since the sequence is bounded, we have that $G_n\rightarrow G$ weakly-star if, and only if, $G_n(x\otimes y)\rightarrow G(x\otimes y)$ for $x\in B_{\ell_p(N)}$ and $y\in B_{\ell_q(M)}$. Take arbitrary $x\in B_{\ell_p(N)}$ and $y\in B_{\ell_q(M)}$. Observe that
$$G_n(x\otimes y)=T_n(i(x)\otimes j(y))\rightarrow T(i(x)\otimes j(y))=G(x\otimes y)$$
where the above convergence follows since $T_n\rightarrow T$ in the $w^*$-topology of $(\ell_p(I)\pten \ell_q(J))^*$. This proves that $G_n\rightarrow G$ in the $w^*$-topology of $(\ell_p(N)\pten \ell_q(M))^*$, as desired.

Now we have that $\ell_p(N)\pten \ell_q(M)$ is isometrically isomorphic to $\ell_p\pten \ell_q$ since $N$ and $M$ are countable. Consequently, $(\ell_p(N)\pten \ell_q(M))^*$ has the $w^*$-Kadec-Klee property by \cite[Theorem 4]{dk}, so $\Vert G_n-G\Vert\rightarrow 0$. This implies that $\Vert T_n-T\Vert\rightarrow 0$. Indeed, given $n\in\mathbb N$ and $x\in B_{\ell_p(I)}, y\in B_{\ell_q(J)}$ we have that
\[\begin{split}
(T_n-T)(x\otimes y)& =(T_n-T)(i(P(x))\otimes j(Q(y)))\\
& =(T_n-T)(i\otimes j)(P(x)\otimes Q(y))\\
& =(G_n-G)(P(x)\otimes Q(y))\\
& \leq \Vert G_n-G\Vert \Vert P(x)\otimes Q(y)\Vert\\
& \leq \Vert G_n-G\Vert.
\end{split}
\]
Since $x,y$ were arbitrary we conclude that $\Vert T_n-T\Vert\leq \Vert G_n-G\Vert$, from where we conclude that $T_n\rightarrow T$ in the norm topology and the proof is finished.
\end{proof}

Now we aim to prove extend the above result when one of the sets is finite. This is immediate from the following result.

\begin{lemma}\label{lemma:inheritkk}
Let $X$ be a Banach space and $Y\subseteq X$ be a $1$-complemented subspace. If $X^*$ has the $w^*$-sequential Kadec-Klee property, then so does $Y$.
\end{lemma}

\begin{proof}
Let $\{y_n^*\}$ be a sequence in $S_{Y^*}$ such that $\{y_n^*\}\rightarrow y^*\in S_{Y^*}$ in the $w^*$-topology. Let us prove that $\Vert y_n^*-y^*\Vert\rightarrow 0$. To this end, take $P:X\longrightarrow Y$ be a norm-one operator with $P(y)=y$ for every $y\in Y\subseteq X$. Since $P^*$ is $w^*-w^*$ continuous we derive that $P^*(y_n^*)\rightarrow P^*(y^*)$ in the $w^*$-topology of $X^*$. Moreover, we claim that $P^*(y_n^*), P^*(y^*)\in S_{X^*}$ for every $n\in\mathbb N$. Let us prove for instance that $\Vert P^*(y^*)\Vert=1$. To this end take $\varepsilon>0$ and take $y\in S_Y$ satisfying that $y^*(y)>1-\varepsilon$. Now we have
$$1-\varepsilon<y^*(y)=y^*(P(y))=P^*(y^*)(y)\leq \Vert P^*(y^*)\Vert.$$
Since $\varepsilon>0$ was arbitrary we conclude that $P^*(y^*)\in S_{Y^*}$.

Since $X^*$ has the $w^*$ sequential Kadec-Klee property we get that $\Vert P^*(y_n^*)-P^*(y^*)\Vert\rightarrow 0$. But, given $n\in\mathbb N$, we have
$$\Vert P^*(y_n^*)-P^*(y^*)\Vert=\sup_{x\in B_X}\Vert (y_n^*-y^*)(P(x))\Vert=\sup_{y\in B_Y}\Vert (y_n^*-y^*)(y)\Vert=\Vert y_n^*-y^*\Vert,$$
from where $\Vert y_n^*-y^*\Vert\rightarrow 0$ and the result follows.
\end{proof}

Now we are ready to give the following result.

\begin{theorem}\label{theo:SSDgenerallplq}
Let $2<p,q<\infty$. For every pair of non-empty sets $I$ and $J$ the space $(\ell_p(I)\pten \ell_q(J))^*$ has the $w^*$-sequential Kadec-Klee property. In particular, $\ell_p(I)\pten \ell_q(J)$ is SSD.
\end{theorem}

\begin{proof} Observe that the case that $I$ and $J$ are finite is trivial since $\ell_p(I)\pten \ell_q(J)$ is finite dimensional in this case. Moreover, the cases when $I$ and $J$ are infinite is proved in Theorem \ref{theo:SSDnoseparable}. In order to prove the result, it is enough to consider the case that just one of them is finite, so assume with no loss of generality that $I$ is finite. We can assume with no loss of generality that $I=\{1,\ldots, n\}\subseteq \mathbb N$ where $n=\operatorname{dim}(\ell_p(I))$.

Observe that $\ell_p(I)$ is a norm-one complemented subspace of $\ell_p$. In particular, $\ell_p(I)\pten \ell_q(J)$ is a norm-one complemented subspace of $\ell_p\pten \ell_q(I)$. Since the latter space has the $w^*$-Kadec-Klee property by Theorem \ref{theo:SSDnoseparable}, the result follows by Lemma \ref{lemma:inheritkk}.

Observe that the consequence on the SSD follows since $\ell_p(I)\pten \ell_q(J)$ is reflexive since every bounded operator $\ell_p(I)\longrightarrow \ell_q(J)^*$ is compact \cite[Theorem A2]{rosenthal} and by \cite[Theorem 4.21]{ryan}.
\end{proof}

Following the notation of \cite[Definition 2.1]{dklm}, given two Banach spaces $X,Y$, we say that the pair $(X,Y)$ has the \textit{$L_{p,p}$ for bilinear mappings} if, given $\varepsilon>0$ and $(x,y)\in S_X\times S_Y$, there exists $\eta>0$ (which depends on $\varepsilon$ and on the pair $(x,y)$) satisfying that if a bilinear mapping $B:X\times Y\longrightarrow \mathbb R$  with $\Vert B\Vert=1$ satisfies $B(x,y)>1-\eta$ then there exists another bilinear mapping $G:X\times Y\longrightarrow \mathbb R$ with $\Vert G\Vert=1$, $G(x,y)=1$ and $\Vert B-G\Vert<\varepsilon$.

It is clear, and explicitly proved in \cite[Proposition 4.2]{djmr}, that if $X\pten Y$ is SSD then the pair $(X,Y)$ has the $L_{p,p}$ for bilinear mappings.

As an immediate application of Theorem \ref{theo:SSDgenerallplq} we obtain the following corollary, which improves \cite[Theorem 2.7 (a)]{dklm}.

\begin{corollary}
Let $2<p,q<\infty$ and $I,J$ be two non-empty sets. Then the pair $(\ell_p(I),\ell_q(J))$ has the $L_{p,p}$ for bilinear mappings.
\end{corollary}

\section{Tensors which do not attain its norm}\label{section:naprojtensor}

One consequence of the isometric identification $\ell_1(I)\pten X=\ell_1(I,X)$ is \cite[Proposition 2.8]{ryan}, which establishes that, given two Banach spaces $X$ and $Y$, then for every $z\in X\pten Y$ and every $\varepsilon>0$, there exist sequences $(x_n)$ in $X$ and $(y_n)$ in $Y$ with $u=\sum_{n=1}^\infty x_n\otimes y_n$ (where the above convergence is in the norm topology of $X\pten Y$) and such that $\Vert z\Vert\leq \sum_{n=1}^\infty \Vert x_n\Vert\Vert y_n\Vert\leq \Vert z\Vert+\varepsilon$. Consequently, it follows that
$$\Vert z\Vert=\inf\left\{\sum_{n=1}^\infty \Vert x_n\Vert\Vert y_n\Vert: \sum_{n=1}^\infty \Vert x_n\Vert\Vert y_n\Vert<\infty,  u=\sum_{n=1}^\infty x_n\otimes y_n \right\}$$
where the infimum is taken over all the possible representations of $u$ as limit of a series in the above form.

According to \cite[Definition 2.1]{djrr21}, an element $z\in X\pten Y$ is said \textit{attain its projective norm} if the above infimum is actually a minimum, that is, if there exists a sequence $(x_n)$ in $X$ and $(y_n)$ in $Y$ such that $\Vert u\Vert=\sum_{n=1}^\infty \Vert x_n\Vert \Vert  y_n\Vert$ and $u=\sum_{n=1}^\infty x_n\otimes y_n$. We denote $\NA_\pi(X\pten Y)$ the set of those $z$ which attains its nuclear norm.

In the papers \cite{dgjr22,djrr21} an intensive study of the structure of $\NA_\pi(X\pten Y)$ is done in connection of how big can this set be. For instance, it is known that $\NA_\pi(X\pten Y)$ is (norm) dense in $X\pten Y$ if $X$ and $Y$ are dual spaces with the Radon-Nikodym property and one of them has the approximation property \cite[Theorem 4.6]{dgjr22} or in the classical Banach spaces \cite[Example 4.12]{djrr21}, but there are examples of Banach spaces $X$ and $Y$ where $\NA_\pi(X\pten Y)$ is not dense \cite[Theorem 5.1]{djrr21}. It is also known that there are examples of $X$ and $Y$ where $\NA_\pi(X\pten Y)=X\pten Y$ like $X,Y$ finite dimensional, $X=\ell_1(I)$ and $Y$ any Banach space, $X$ a finite dimensional polyhedral and $Y$ any dual Banach space or $X=Y$ being a complex Hilbert space (see \cite[Propositions 3.5, 3.6 and 3.8]{djrr21} and \cite[Theorem 4.1]{dgjr22}).

In general, few is known about when a particular element $z\in X\pten Y$ does (or does not) attain its nuclear norm, and a manifestation of this is that, in all the examples $X\pten Y$ where there exists an element $z$ not attaining its projective norm, no explicit description of such $z$ is given and the conclusion is obtained by an indirect argument like an argument of non-density of norm-attaining bilinear mapping \cite[Example 3.12 (b), (c) and (d)]{djrr21} or the existence of a bilinear form which attains its norm as a functional on $X\pten Y$ but which does not attains its norm as a bilinear mapping \cite[Example 3.12 (a)]{djrr21}.

In the following we will exhibit examples of tensor product spaces for which the norm attaining tensors are finite linear combination of basic tensors. For the establishment of the theorem we need a bit of notation. Recall that the norm of a Banach space $X$ is said to \textit{locally depend upon finitely many coordinates} if for every $x\in X\setminus \{0\}$, there exists $\varepsilon>0$, a subset $\{f_1,\ldots, f_N\}\subseteq X^*$ and a continuous function $\varphi:\mathbb R^N\longrightarrow \mathbb R$ satisfying that $\Vert y\Vert=\varphi(f_1(y),\ldots, f_N(y))$ for $y\in X$ with $\Vert y-x\Vert<\varepsilon$.

Clearly, this property is inherited by closed subspaces. We refer to \cite{godefroy,hz} and references therein for background. For instance, closed subspaces of $c_0$ have
this property [17, Proposition III.3]. Conversely, every infinite dimensional Banach space whose norm locally depends upon finitely many coordinates contains an isomorphic copy of
$c_0$ [17, Corollary IV.5].

Now we are ready to present the following theorem.

\begin{theorem}\label{theo:nonnormattainfincor}
Let $X$ be an infinite dimensional Banach space whose norm depends upon finitely many coordinates and let $Y$ be an infinite dimensional Hilbert space. Then
$$\NA_\pi(X\pten Y)\subseteq X\otimes Y.$$
In particular, there are tensors in $X\pten Y$ which do not attain its projective norm.
\end{theorem}

\begin{proof}
Take $z\in \NA_\pi(X\pten Y)$ with $\Vert z\Vert=1$, and let us prove that we can write $z$ as a finite sum of basic tensors. To this end, since $z\in \NA_\pi(X\pten Y)$ we can write $z=\sum_{n=1}^\infty \lambda_n x_n\otimes y_n$ for suitable $x_n\in S_X, y_n\in S_Y$ and $\lambda_n\in ]0,1]$ for every $n\in\mathbb N$ with $\sum_{n=1}^\infty \lambda_n=1$. Take $T\in S_{L(X,Y^*)}$ with $T(z)=1$. A convexity argument implies that $T(x_n)(y_n)=1$ for every $n\in\mathbb N$. Observe that, since $Y$ is a Hilbert space we have, under the natural identification $Y=Y^*$, that $T(x_n)(y_n)=1$ implies $y_n=T(x_n)\in T(X)$. On the other hand, since $Y^*$ is strictly convex and $T$ attains its norm we conclude that $T$ has finite rank \cite[Lemma 2.8]{martin16}. This implies that $y_n$ lives in the finite dimensional subspace $T(X)$ of $Y$. Let us conclude form here the desired result.

Take an orthonormal basis $\{v_1,\ldots, v_q\}$ of the Hilbert space $T(X)$. Since $y_n\in T(X)$ we can write $y_n:=\sum_{i=1}^q \alpha_i^n v_i$ with $\Vert y_n\Vert^2=1=\sum_{i=1}^n (\alpha_i^n)^2$. By H\"older inequality we conclude that
$$\sum_{i=1}^q \vert\alpha_i^n\vert\leq \sqrt{q} \left(\sum_{i=1}^q (\alpha_i^n)^2\right)^\frac{1}{2}=\sqrt{q}$$
holds for every $n\in\mathbb N$. Now, given $k\in\mathbb N$, we have
$$\sum_{n=1}^k \lambda_n x_n\otimes y_n=\sum_{n=1}^k \lambda_n x_n\otimes \left(\sum_{i=1}^q\alpha_i^n v_i\right)=\sum_{i=1}^q \left(\sum_{n=1}^k \alpha_i^n \lambda_n x_n  \right)\otimes v_i.$$
Observe that the above sequence in $k\in\mathbb N$ converges to $z$ when $k\rightarrow \infty$. On the other hand, given $1\leq i\leq q$ we have that the sequence $\sum_{n=1}^k \alpha_i^n \lambda_n x_n$ in $k$ converges in norm to some element of $X$. In order to show this it is enough, by the completeness of $X$, to prove that the series is absolutely convergent. But this is immediate since, given $k\in\mathbb N$, we have
$$\sum_{n=1}^k \left\Vert \alpha_i^n \lambda_n x_n\right\Vert=\sum_{n=1}^k \vert \alpha_i^n\vert \lambda_n \Vert x_n\Vert=\sum_{n=1}^k \vert \alpha_i^n\vert \lambda_n\leq \sqrt{q}\sum_{n=1}^k \lambda_n\leq  \sqrt{q}\sum_{n=1}^\infty \lambda_n =\sqrt{q}.$$
To shorten, let us write $a_i^k:=\sum_{n=1}^k \alpha_i^n \lambda_n x_n$. We know that $a_i^k\rightarrow a_i$ in norm for some $a_i\in X$. It is immediate that $a_i^k\otimes v_i\rightarrow a_i\otimes v_i$ in the norm topology of $X\pten Y$. By linearity of the limit we have that $\sum_{i=1}^q a_i^k\otimes v_i\rightarrow \sum_{i=1}^q a_i\otimes v_i$. However, the above sequence converges to $z$. By the uniqueness of limit we conclude 
$$z=\sum_{i=1}^q a_i\otimes v_i,$$
which proves that $z\in X\otimes Y$, as desired.

To conclude the last part, observe that the projective norm is not complete on $X\otimes Y$ since $X$ and $Y$ are infinite dimensional. Consequently, there exists $z\in X\pten Y\setminus X\otimes Y$. By the above, $z$ can not attain its projective norm.
\end{proof}

\begin{remark}\label{remark:densicomplement}\begin{enumerate}

\item In the hypothesis of Theorem \ref{theo:nonnormattainfincor} it is easy to construct elements which do not attain its projective norm. For instance, take $\{e_n\}\subseteq Y$ an infinite orthonormal set and take $\{x_n\}\subseteq S_X$ being linearly independent. Then $z:=\sum_{n=1}^\infty \frac{1}{2^n}x_n\otimes e_n$ does not attain its projective norm because it can not be written as a finite sum of basic tensors. Indeed, observe that since $Y$ is a Hilbert space, we have that $X\pten Y$ is precisely the space of nuclear operators $N(Y,X)$ (see \cite[Corollary 4.8]{ryan} for details). If we see $z$ an operator $T_z: Y\longrightarrow X$ by $T(y):=\sum_{n=1}^\infty \langle e_n,y\rangle x_n$, we have that $T$ is not a finite rank operator since $T(X)$ contains $T(e_n)=\frac{1}{2^n} x_n$, so $\{x_n: n\in\mathbb N\}\subseteq T(X)$, which implies that $T(X)$ is infinite dimensional. 

To the best of our knownledge, this is the first explicit description of an element in a projective tensor product which does not attain its projective norm.

\item In \cite[Theorem 4.2]{dgjr22} it is prove that if $X$ is a finite dimensional polyhedral space and $Y$ is a dual space, then $\NA_\pi(X\pten Y)=X\pten Y$. Observe that this result is false if $X$ is infinite dimensional, as $c_0\pten \ell_2$ shows by Theorem \ref{theo:nonnormattainfincor} and by the fact that there are elements in $c_0\pten \ell_2\setminus c_0\otimes \ell_2$.

\item Observe that given an infinite dimensional Banach space $X$ whose norm depends upon finitely many coordinates and given an infinite dimensional Hilbert space $Y$, Theorem \ref{theo:nonnormattainfincor} reveals that $X\pten Y\setminus \NA_\pi(X\pten Y)$ is non-empty. However, it follows that the set of non-norm attaining elements is even dense in $X\pten Y$. 

In order to prove it, take $z\in X\pten Y$, and let us approximate it by non-norm attaining elements. If $z\in X\pten Y\setminus X\otimes Y$ (i.e. if $z$ can not be written as finite sums of basic tensors) then $z\notin \NA_\pi(X\pten Y)$ and there is nothing to prove.

Otherwise, if $z\in X\otimes Y$, we can select $v\in X\pten Y\setminus X\otimes Y$ since $X$ and $Y$ are infinite dimensional Banach space. Then $z+\frac{1}{n}v\rightarrow z$. Moreover, $z+\frac{1}{n}v$ does not attain its norm for every $n$ because $z+\frac{1}{n}v\notin X\otimes Y$ (otherwise $v=n(z+\frac{1}{n}v-z)\in X\otimes Y$ since $X\otimes Y$ is a vector space).

Since $z$ was arbitrary we conclude that $(X\pten Y)\setminus \NA_\pi(X\pten Y)$ is dense.\end{enumerate}\end{remark}

The above point (3) in Remark \ref{remark:densicomplement} together with \cite[Theorem 4.8]{djrr21} allows us to obtain a number of spaces $X\pten Y$ where both $\NA_\pi(X\pten Y)$ and $X\pten Y\setminus \NA_\pi(X\pten Y)$ are dense. Before that, let us introduce the following notation from \cite{cas}: A Banach space $X$ is said to have the {\it metric $\pi$-property} if given $\varepsilon>0$ and $\{x_1,\ldots, x_n\}\subseteq S_X$ a finite collection in the sphere, then we can find a finite dimensional 1-complemented subspace $M\subseteq X$ such that for each $i \in \{1, \ldots, n \}$ there exists $x_i'\in M$ with $\Vert x_i-x_i'\Vert<\varepsilon$. See \cite[Example 4.12]{djrr21} for examples of Banach spaces with the metric $\pi$-property.

\begin{theorem}\label{theo:densiyintevacio}
Let $X$ be a Banach space with the metric $\pi$-property and whose norm locally depends upon finitely many coordinates (in particular $X=c_0(I)$) and $Y$ be a Hilbert space. Then both $\NA_\pi(X\pten Y)$ and $X\pten Y\setminus \NA_\pi(X\pten Y)$ are dense in $X\pten Y$. 
\end{theorem}

\textbf{Acknowledgements:} The author thanks S. Dantas, M. Jung, M. Mazzitelli and J. Tom\'as Rodr\'iguez for sharing an unpublished version of the preprint \cite{djmr}. He also thanks L. Garc\'ia-Lirola and M. Mazzitelli for useful conversations on the topic of the paper.

\end{document}